\documentclass[reqno,12pt]{amsart}
\usepackage{bbm}
\usepackage[all,pdf]{xy}
\usepackage{epsfig}
\usepackage{amsmath}
\usepackage{amssymb}
\usepackage{amscd}
\usepackage{graphicx}
\usepackage{color}
\usepackage{mathptmx}
\usepackage{ulem}
\allowdisplaybreaks[4]

\makeatletter
\@namedef{subjclassname@2020}{%
	\textup{2020} Mathematics Subject Classification}
\makeatother
\usepackage[colorlinks=true]{hyperref}
\usepackage{amsfonts}

\usepackage[top=25mm, bottom=27mm, left=27mm, right=27mm]{geometry}

\newtheorem{thm}{Theorem}[section]
\newtheorem{lemma}[thm]{Lemma}

\newtheorem{cl}{Claim}
\newtheorem{prop}[thm]{Proposition}

\newtheorem{ques}[thm]{Question}
\newtheorem{cor}[thm]{Corollary}

\newtheorem{rem}[thm]{Remark}

\newcommand{\norm}[1]{\left\Vert #1\right\Vert}

\def \N {\mathbb N}

\def \Z {\mathbb Z}
\def \R {\mathbb R}

\def \E {\mathbb E}

\def\B {\mathcal B}

\def \X {\mathcal{X}}
\def \Y {\mathcal{Y}}

\def \ep {\epsilon}

\numberwithin{equation}{section}

\begin{document}

\title[]{Monochromatic polynomial sumset structures on $\N$}
	
	\author[]{Zhengxing Lian and Rongzhong Xiao}
	
	\address[Zhengxing Lian]{School of Mathematical Sciences, Xiamen University, Xiamen, Fujian, 361005, PR China}
	\email{lianzx@xmu.edu.cn}
	
	\address[Rongzhong Xiao]{School of Mathematical Sciences, University of Science and Technology of China, Hefei, Anhui, 230026, PR China}
	\email{xiaorz@mail.ustc.edu.cn}

	
\begin{abstract}
In the paper, we search for monochromatic infinite additive structures involving polynomials over $\N$. It is proved that for any $r\in \N$, any two distinct natural numbers $a,b$, and any $2$-coloring of $\N$, there exist two sets $B,C\subset \N$ with $|B|=r$ and $|C|=\infty$ such that there exists some color containing $B+aC$ and $B+bC$.
\end{abstract}

\maketitle

\section{Introduction}
\subsection{Infinite sumsets structures}
In 1974, Hindman \cite{H74} proved a conjecture of Graham and Rothschild \cite{GrahamRothschild71} by showing the following statement: For any finite coloring of $\N$, there exists a sequence $\{x_n\}_{n\ge 1}$ in $\N$ such that $$\{\sum_{n\in F}x_{n}:F\ \text{is a non-empty finite subset of}\ \N\}$$ is monochromatic. 

Hindman's theorem implies that for any finite coloring of $\N$, there exists a sequence $\{B_n\}_{n\ge 1}$ of infinite subsets of $\N$ such that 
$$\bigcup_{k\ge 1}(B_1+\cdots+B_k)$$ is monochromatic. Naturally, ones seek to search for its density analogues. Very recently, Hern\'andez, Kousek, and Radi\'c \cite{HKR2025} confirmed this by showing that for any $A\subset \N$ with positive upper Banach density, there exists a sequence $\{B_n\}_{n\ge 1}$ of infinite subsets of $\N$ such that 
$$\bigcup_{k\ge 1}(B_1+\cdots+B_k)\subset A.$$

For more on infinite sumsets structures, see the survey \cite{KMRR2025} and these references \cite{MRR,H19,KMRR22a,KMRR22b,KousekRadic2025,Kousek2025a,KMRR2025b,Her2025}.
\subsection{Polynomial progressions in ``large" sets}
In 1927, van der Waerden \cite{VDW} proved that for any finite coloring of $\N$, there exists some color containing arithmetic progressions of any given length. In 1975, Szemer\'edi \cite{S75} proved a conjecture of Erd\"os and Tur\'an by showing that every $A\subset \N$ with positive upper density contains arithmetic progressions of any given length. In 1977, Furstenberg \cite{F77} provided an ergodic theoretic proof for Szemer\'edi theorem. Following Furstenberg’s approach, in 1996, Bergelson and Leibman built polynomial extensions of van der Waerden's theorem and Szemer\'edi's by showing the following multiple recurrence:
\begin{thm}\label{thm1-1}
	$($\cite[Theorem A]{BL96}$)$ Given $k\in \N$, let $P_1,\ldots,P_k\in \Z[n]$ with $P_{1}(0)=\cdots=P_{k}(0)=0$. Then for any measure preserving system $(X,\X,\mu,T)$ and any $A\in \X$ with $\mu(A)>0$, 
	$$\liminf_{N\to\infty}\frac{1}{N}\sum_{n=1}^{N}\mu(A\cap T^{-P_{1}(n)}A\cap \cdots \cap T^{-P_{k}(n)}A)>0.$$ 
\end{thm}
By using Theorem \ref{thm1-1} iteratively and Furstenberg's correspondence principle \cite{F77}, we have the following:
\begin{prop}\label{prop1-1}
	Given $d\in \N$, let $\mathcal{A}=\{P_1,\ldots,P_d\}\subset \Z[n]$ and for each $1\le i\le d$, $P_i$ has positive leading coefficient and zero constant term. Then for any $r\in \N$ and any $A\subset \N$ with positive upper Banach density, there exist two sets $B,C\subset \N$ with $|B|=\infty$ and $|C|=r$ such that the set $$\bigcup_{P\in \mathcal{A}} (B+P(C))\subset A.$$
\end{prop} 
Although the establishment process of the above proposition is simple, it combines aspects of  polynomial version of Szemer\'edi's theorem and infinite sumsets structures in some sense.
\subsection{Main results}
In this paper, we seek to search for certain additive structures in the spirit of Proposition \ref{prop1-1}. Specifically, we try to find additive structures of the form
\begin{equation}\label{eq0}
	\bigcup_{P\in \mathcal{A}} (B+P(C))
\end{equation}
in some large sets, where $|C|=\infty$.

The following example demonstrates that the configuration \eqref{eq0} may not exist in some set with positive upper Banach density: Let $A=\bigcup_{n\ge 1}\{n^2+n,\ldots,n^2+2n\}$. Then for any $B,C\subset \N$ with $|C|=\infty$, $B+C^2$ is not a subset of $A$.

Due to the above example, we have to search for the configuration \eqref{eq0} in the finite coloring setting.
\begin{ques}\label{Q1}
	Let $d\in \N,k>1$ and $r\in \N\cup \{\infty\}$. Let $\mathcal{A}=\{P_1,\ldots,P_d\}\subset \Z[n]$ and for each $1\le i\le d$, $P_i$ has positive leading coefficient and zero constant term. Is it true that for any $k$-coloring of  $\N$, there exist two sets $B,C\subset \N$ with $|B|=r$ and $|C|=\infty$ such that the set $$\bigcup_{P\in \mathcal{A}} (B+P(C))$$ is monochromatic ?
\end{ques}
\begin{rem}
	\begin{itemize}
		\item[(1)] By Ramsey theorem (see \cite[Theorem 5 of Chapter 1]{GRS}), we have that when $d=1$, the answer to Question \ref{Q1} is positive.
		\item[(2)] When $r\in \N$ and $C$ is a finite set, Question \ref{Q1} is a special case of Proposition \ref{prop1-1}.
	\end{itemize}
\end{rem}
Unfortunately, for the most part, the answer to Question \ref{Q1} is negative. This can be shown be the following counter-examples.
\begin{prop}\label{P1}
	Let $P,Q\in \Z[n]$ with positive leading coefficient and zero constant term and $P\neq Q$. Then:
	\begin{itemize}
		\item[(1)] If $\max\{\deg P,\deg Q\}>1$, there exists a $2$-coloring $\phi:\N\rightarrow \{1,2\}$ such that for any $i\in \{1,2\}$ and any $n\in \N$, the set $$\{m\in \N:\phi(n+P(m))=\phi(n+Q(m))=i\}$$ is empty or finite.
		\item[(2)] If $\deg P=\deg Q=1$, there exists a $3$-coloring $\phi:\N\rightarrow \{1,2,3\}$ such that for any $i\in \{1,2,3\}$ and any $n\in \N$, the set $$\{m\in \N:\phi(n+P(m))=\phi(n+Q(m))=i\}$$ is empty or finite.
		\item[(3)] If $\deg P=\deg Q=1$, there exists a $2$-coloring $\phi:\N\rightarrow \{1,2\}$ such that for any $i\in \{1,2\}$, there are not two infinite subsets $B,C\subset \N$ such that $$\phi(B+P(C))=\phi(B+Q(C))=i.$$
	\end{itemize}
\end{prop}
\begin{prop}\label{P2}
	For any three distinct natural numbers $a,b,c$, there exists a $2$-coloring $\phi:\N\rightarrow \{1,2\}$ such that for any $i\in \{1,2\}$ and any $n\in \N$, the set $$\{m\in \N:\phi(n+am)=\phi(n+bm)=\phi(n+cm)=i\}$$ is empty or finite.
\end{prop}
Fortunately for some though, for the rest case, the answer to Question \ref{Q1} is positive.
\begin{thm}\label{T1}
	For any $r\in \N$, any two distinct natural numbers $a,b$ and any $2$-coloring of $\N$, there exist two sets $B,C\subset \N$ with $|B|=r$ and $|C|=\infty$ such that there exists some color which contains $B+aC$ and $B+bC$.
\end{thm}
Clearly, for any $2$-coloring of $\N$, there must be a cell which is ``large" on numerous occasions. Rough speaking, the property of $2$-colorings ensures the existence of the monochromatic configuration 
$(B+aC)\cup (B+bC)$.

\subsection*{Structure of the paper.} In Section \ref{SP}, we recall some notions and results. In Section \ref{SA}, we prove Theorem \ref{T1}. In Section \ref{SC}, we show Propositions \ref{P1}, \ref{P2}. 
\subsection*{Acknowledgement}
The first author is partially supported by NNSF of China (12101517, 12171400) and Fujian Natural Science Foundation (2023J05009). The second author is supported by NNSF of China (123B2007, 12371196). Two authors' thanks go to Wen Huang, Song Shao and  Xiangdong Ye for their useful suggestions, and Florian Richter for informing us the fact that Ramsey theorem implies that when $d=1$, the answer to Question \ref{Q1} is positive.
\section{Preliminaries}\label{SP}
\subsection{Notation}
\begin{itemize}
	\item For a topological space $X$, $\mathcal{B}(X)$ denotes the Borel $\sigma$-algebra on $X$.
	\item Let $A$ be a non-empty subset of $\N$. The \textbf{upper Banach density} of $A$ is defined as $$\limsup_{N-M\to\infty}\frac{|A\cap \{M+1,\ldots,N\}|}{N-M}.$$
\end{itemize}
\subsection{A Gowers's result}
In \cite{Gowers}, Gowers built the effective version of Szemer\'edi's theorem on arithmetic progressions.
\begin{thm}\label{CT1}
	$($\cite[Theorem 1.3]{Gowers}$)$ For each $k\in \N$ and each sufficiently large natural number $N$, every subset of $\{1,\ldots,N\}$ with at least $$N(\log \log N)^{-2^{-2^{k+9}}}$$ elements contains an arithmetic progression of length $k$.
\end{thm}
\subsection{Bergelson's finite intersection theorem}
In the proof of Theorem \ref{T1}, we need the following result.
\begin{thm}\label{CT2}
	$($\cite[Theorem 1.1]{B85}$)$ Let $\{E_n\}_{n\ge 1}$ be a sequence of events in a probability space $(X,\X,\mu)$ with $$\inf_{n\ge 1}\mu(E_n)>0.$$ Then there exists a sequence $\{i_n\}_{n\ge 1}\subset \N$ such that for any $k\in \N$, one has $$\mu(E_{i_1}\cap \cdots \cap E_{i_k})>0.$$ In fact, one can take the set $\{i_1,i_2,\ldots\}$ to have positive upper Banach density.
\end{thm}

\subsection{Conditional expectation and disintegration of a measure}
Let $\pi:X\to Y$ be a measure preserving transformation between two Lebesgue probability spaces $(X,\X,\mu)$ and $(Y,\Y,\nu)$. For any $f\in L^{1}(\mu)$, the \textbf{conditional expectation of $f$ with respect to $Y$} is the function $\E_{\mu}(f|Y)$, defined in $L^{1}(\mu)$, such that for any $A\in \Y$, $$\int_{\pi^{-1}A}fd\mu=\int_{\pi^{-1}A}\E_{\mu}(f|Y)d\mu.$$
Clearly, there exists an $\tilde{f}\in L^{1}(\nu)$ such that $\tilde{f}\circ \pi=\E_{\mu}(f|Y)$ almost everywhere. So, we view $\E_{\mu}(f|Y)$ as an element of $L^{1}(\nu)$ sometimes.

Moreover, there exists a unique $\Y$-measurable map $Y\to \mathcal{M}(X,\X),y\mapsto \mu_y$, called the \textbf{disintegration of $\mu$ with respect to $Y$}, under neglecting $\nu$-null sets such that for any $f\in L^{\infty}(\mu)$, $$\E_{\mu}(f|Y)(x)=\int_{X}fd\mu_{\pi(x)}$$ for $\mu$-a.e. $x\in X$, where $\mathcal{M}(X,\X)$ is the collection of probability measures on $(X,\X)$, endowed with standard Borel structure.
\subsection{Measure preserving systems, factors and joingings}
A tuple $(X,\X,\mu,T)$ is a \textbf{measure preserving system} if $(X,\X,\mu)$ is a Lebesgue probability space and $T:X\to X$ is an invertible measure preserving transformation. We say that $(X,\X,\mu,T)$ is \textbf{ergodic} or $\mu$ is \textbf{ergodic} (under $T$) if only the $T$-invariant subsets in $\X$ have measure $0$ or $1$.

Two measure preserving systems $(X,\X,\mu,T)$ and $(Y,\Y,\nu,S)$ are \textbf{isomorphic} if there exists an invertibe measure preserving transformation $\Phi:X_0\to Y_0$ with $\Phi \circ T=S\circ\Phi$, where $X_0$ is a $T$-invariant full measurable subset of $X$ and $Y_0$ is an $S$-invariant full measurable subset of $Y$.

A \textbf{factor} of a measure preserving system $(X,\X,\mu,T)$ is a $T$-invariant sub-$\sigma$-algebra of $\X$. A \textbf{factor map} from $(X,\X,\mu,T)$ to $(Y,\Y,\nu,S)$ is a measurable map $\pi:X_{0}\rightarrow Y_{0}$ with $\pi\circ T=S\circ\pi$ and such that $\nu$ is the image of $\mu$ under $\pi$, where $X_0$ is a $T$-invariant full measurable subset of $X$ and $Y_0$ is an $S$-invariant full measurable subset of $Y$. In this case, $\pi^{-1}(\mathcal{D})$ is a factor of $(X,\X,\mu,T)$. And every factor of $(X,\X\mu,T)$ can be obtained in this way.

In this paper, we focus on the \textbf{Kronecker factor} of $(X,\X,\mu,T)$, which is the sub-$\sigma$-algebra of $\X$ spanned by the eigenfunctions of $X$ with respect to $T$. When $(X,\X,\mu,T)$ is ergodic, by \cite[Proposition 4.13.(iv)]{HK}, its Kronecker factor can be viewed as the measure preserving system $(Z,\B(Z),m,R)$, where $Z$ is a compact abelian group, $m$ is the Haar measure on $Z$, and $R$ is a translation on $Z$.

In this paper, we need to use some properties of Kronecker factors.
\begin{lemma}\cite[Theorem 4.16]{HK}\label{eigenfunctions and invariant}
Let $(X,\X,\mu,T)$ be a measure preserving system. Then the subspace of $L^2(\mu\times \mu)$ consisting of $T\times T$-invariant functions is the subspace spanned by functions of the form $\overline{f}_1\otimes f_2$, where each $f_i\in L^2(X,\mu),i=1,2$ is an eigenfunction of $X$ with respect to $T$ such that the corresponding eigenvalues are equal.
\end{lemma}
By the above lemma, one can immediately give the following corollary.
\begin{cor}\label{orthogonal with eigenfunctions} Let $(X,\X,\mu,T)$ be a measure preserving system. Let $\psi\in L^2(\mu)$ and it is orthogonal to all eigenfunctions of $X$ with respect to $T$. Then for the sub-$\sigma$-algebra $\mathcal{I}$ of $\X\otimes \X$ spanned by all $T\times T$-invariant functions, one has that $$\norm{\mathbb{E}_{\mu\times \mu}(\psi\otimes \overline{\psi}|\mathcal{I})}_{2}=0.$$
\end{cor}
\begin{proof}
By Lemma \ref{eigenfunctions and invariant}, it is left to show that $$\int (\psi\otimes \overline{\psi} )\cdot (\overline{f}_1\otimes f_2)d\mu\times \mu=0, $$ where each $f_i\in L^2(X,\mu),i=1,2$ is an eigenfunction of $X$ with respect to $T$ such that the corresponding eigenvalues are equal.
By the assumption on $\psi$, we have that $\int \psi\cdot fd\mu=\int \overline{\psi}\cdot fd\mu=0$, where $f\in L^2(\mu)$ is an eigenfunction of $X$. This finishes the proof.
\end{proof}

The following known statement gives a property of the Kronecker factor. It can be deduced from a similar argument of \cite[Second paragraph of the statements after Definition 3.3]{F77}.
\begin{lemma}\label{lem6}
	Let $(X,\X,\mu,T)$ be a measure preserving system. Then for any $k\in\N$, the Kronecker factor of $(X,\X,\mu,T)$ is isomorphic to one of $(X,\X,\mu,T^k)$.
\end{lemma}

A \textbf{joining} of two measure preserving systems $(X,\X,\mu,T)$ and $(Y,\Y,\nu,S)$ is a probability measure $\lambda$ on $X\times Y$, invariant under $T\times S$ and whose projections on $X$ and $Y$ are $\mu$ and $\nu$, respectively. If $\lambda$ is ergodic under $T\times S$, then we say that $\lambda$ is an ergodic joining. Analogously, we can define joining on more measure preserving systems.

\subsection{Topological dynamical systems and invariant measures}
A pair $(X,T)$ is a \textbf{topological dynamical system} if $X$ is a compact metric space and $T:X\to X$ is a homeomorphism. If there is no non-empty proper closed $T$-invariant subset of $X$, we say that $(X,T)$ is \textbf{minimal}.
For any $x\in X$, we denote its orbit $\{T^{n}x:n\in\Z\}$ by $\text{orb}(x,T)$. We say that $x$ is a \textbf{minimal point} of $(X,T)$ if $(\overline{\text{orb}(x,T)},T)$ is minimal.

The following two results characterize some properties of minimal points.
\begin{lemma}\label{lem5}
	 $($\cite[Exercises 1.1.2]{G-book}, \cite{GH}$)$ Let  $(X,T)$ be a topological dynamical system and $x\in X$. $x$ is a minimal point of $(X,T)$ if and only if for any open neighborhood $U$ of $x$, $\{n\in \Z:T^{n}x\in U\}$ is syndetic\footnote{Let $A$ be a non-empty subset of $\Z$ (resp. $\N$). $A$ is \textbf{syndetic} if and only if there exists a non-empty finite subset $E$ of $\Z$ such that $A+E=\Z$ (resp. $A+E=\N$).}.
\end{lemma}
\begin{lemma}\label{lem4}
$($\cite[Theorem 8.7]{F-book}$)$ Let  $(X,T)$ be a topological dynamical system. Then for any $x\in X$, there exist a sequence $\{n_i\}_{i\ge 1}\subset \N$ and a minimal point $y$ of $(X,T)$ such that $$T^{n_i}x\to y$$ as $i\to\infty$.
\end{lemma}

Let  $(X,T)$ be a topological dynamical system. By Krylov-Bogolioubov Theorem, we have that the following two sets $$\mathcal{M}(X,T):=\{\mu:\mu\ \text{is}\ \text{a}\ \text{Borel}\ \text{probability}\ \text{measure}\ \text{on}\ \text{X}\ \text{and}\ \text{it}\ \text{is}\ \text{invariant}\ \text{under}\ T\}$$ and $\mathcal{M}^{e}(X,T):=\{\mu:\mu\in \mathcal{M}(X,T)\ \text{and}\ \text{it}\ \text{is}\ \text{ergodic}\ \text{under}\ T\}$ are non-empty.
Let $\mu\in \mathcal{M}^{e}(X,T)$ and $y\in X$. We say that $y$ is a \textbf{generic point} of $\mu$ if for any $f\in C(X)$, we have that  $$\lim_{N\to\infty}\frac{1}{N}\sum_{n=1}^{N}f(T^{n}y)=\int_{X}fd\mu.$$ By Birkhoff's ergodic theorem, $$\mu(\{y\in X:y\ \text{is}\ \text{a}\ \text{generic}\ \text{point}\ \text{of}\ \mu\})=1.$$
\subsection{Symbolic systems}
Let $\Sigma$ be a finite alphabet with $m$ symbols, $m \ge 2$. Let $\Omega=\Sigma^{\Z}$ be the
set of all sequences $x=\cdots x(-1)x(0)x(1) \cdots=(x(i))_{i\in \Z}$, $x(i) \in
\Sigma$, $i \in \Z$. 
Let $\Sigma$ be equipped with the discrete topology and let $\Omega$ be equipped with the product topology. A compatible metric $d$ on $\Omega$ can be
given by $$d(x,y)=\frac{1}{1+k},\ \text{where}\ k=\min \{|n|:x(n) \not= y(n)
\},x,y \in \Omega.$$ The shift map $T: \Omega \longrightarrow \Omega$
is defined by $(T x)(n) = x(n+1)$ for all $x\in \Omega,n \in \Z$. The
pair $(\Omega,T)$ is called a \textbf{two-sided shift}.

\section{Proof of Theorem \ref{T1}}\label{SA}
In this section, we prove Theorem \ref{T1} via some methods from ergodic theory and combinatorics.

First, we introduce some notation.
Let $X=\{1,2\}^{\Z}$ and $(X,T)$ be a two-sided shift. Each element $x$ of $X$ can be written as $(x(n))_{n\in \Z}$. For any $2$-coloring of $\N$, there exists an $x\in X$ such that the map $\N\to \{1,2\},n\mapsto x(n)$ is equal to the coloring.
Fix $a,b\in\N$ with $a<b$. Let $\tau=T\times T,\sigma=T^{a}\times T^{b}$. Let $U=(E_1\times E_1)\cup (E_2\times E_2)$, where $E_1=\{x\in X:x(0)=1\}\ \text{and}\ E_2=\{x\in X:x(0)=2\}$. For any $(y,z)\in X\times X$, let $N_{\sigma}(U,(y,z))=\{n\in \N: \sigma^{n}(y,z)\in U\}$. 

By the pigeonhole principle, we have that Theorem \ref{T1} is equivalent to the following.
\begin{thm}\label{T3-1}
	For any $r\in \N$ and any $x\in X$, there exist two sets $B,C\subset \N$ with $|B|=r,|C|=\infty$  such that for any $h\in B,k\in C$, $\tau^{h}\sigma^{k}(x,x)\in U$.
\end{thm}

Now, for the topological dynamical system $(X\times X,\sigma)$, we provide a lemma, which classifies all minimal points into two classes.
\begin{lemma}\label{lem1}
	Let $(y,z)$ be a minimal point of the topological dynamical system $(X\times X,\sigma)$. Then $(y,z)$ must satisfy one of the following:
	\begin{itemize}
		\item There exists $d\in \N$ such that for any $k\in \N$,
		$$y(d)\neq z(d),y(d)=y(d+a(b-a)k),\ \text{and}\ z(d)=z(d+b(b-a)k).$$
		\item For any $m\in \N$, either $$N_{\sigma}(U,\tau^{m}(y,z))\ \text{or}\ N_{\sigma}(U,\tau^{m+ab}(y,z))$$ is syndetic.
	\end{itemize}
\end{lemma}
\begin{proof}
	Note that for any $l\in \N$, $\tau^{l}(y,z)$ is still a minimal point of the topological dynamical system $(X\times X,\sigma)$. So, by Lemma \ref{lem5}, we know that for any $l\in \N$, $N_{\sigma}(U,\tau^{l}(y,z))$ is either syndetic or empty.
	
	Assume that there exists an $m\in \N$ such that $$N_{\sigma}(U,\tau^{m}(y,z))\ \text{and}\ N_{\sigma}(U,\tau^{m+ab}(y,z))$$ are empty. As $U=(E_1\times E_1)\cup (E_2\times E_2)$, for any $n\in\N$, we have that
	\begin{equation}\label{eq5}
		y(m+an)\neq z(m+bn)\ \text{and}\ y(m+ab+an)\neq z(m+ab+bn).
	\end{equation}
	By replacing $n$ by $n+a$ and $n+b$ in \eqref{eq5}, we know that for any $n\in \N\cup \{0\}$, $$y(m+ab+an)\neq z(m+bn+b^2)\ \text{and}\ y(m+an+a^2)\neq z(m+ab+bn).$$
	Note that for any $x\in X,t\in{\N}$, $x(t)\in\{1,2\}$. So, for any $n\in \N\cup \{0\}$, we know that $$z(m+ab+bn)=z(m+bn+b^2)=z(m+ab+b(n+b-a))\ \text{and}$$ $$y(m+ab+an)=y(m+an+a^2)=y(m+ab+a(n-(b-a))).$$
	Let $d=m+ab$ and take $n$ from $\{0,b-a,2(b-a),\ldots\}$. Then for any $k\in \N$,
	$$y(d)=y(d+a(b-a)k)\ \text{and}\ z(d)=z(d+b(b-a)k).$$ Note that $$y(d)\neq z(m+b^2)\ \text{and}\ z(m+b^2)=z(m+ab+b(b-a))=z(d).$$ Therefore, $y(d)\neq z(d)$. The proof is complete.
\end{proof}
	By Lemma \ref{lem1}, the proof of Theorem \ref{T3-1} is divided into two cases. For the first case, we give an elementary proof by Theorem \ref{CT1}. For another case, we use some methods from ergodic theory to deal with it.
\subsection{}
In this subsection, we consider the first case of Theorem \ref{T3-1}. Before stating specific result, we introduce a simple fact.
\begin{lemma}\label{lem7}
	Given $k\in \N$, let $Y=\{1,\ldots,k\}^{\Z}$. Then for any non-empty finite subset $F$ of $\Z$ and any non-empty subset $A$ of $Y$, we have that
	$$|\{(y(i))_{i\in F}\in \{1,\ldots,k\}^{F}:y\in A\}|\le k^{|F|}.$$
\end{lemma}

For each $x\in X$, we use $\text{Min}((x,x),\sigma)$ to denote all minimal points of the topological dynamical system $(\overline{\text{orb}((x,x),\sigma)},\sigma)$. 
\begin{prop}\label{pro1}
	Let $x\in X$. If there is $(y,z)\in \text{Min}((x,x),\sigma)$ such that the following hold:
	\begin{itemize}
		\item There exists $\tilde{d}\in \N$ such that for any $\tilde{k}\in \N$,
		\begin{equation}\label{value of (y,z)}y(\tilde{d})\neq z(\tilde{d}),y(\tilde{d})=y(\tilde{d}+a(b-a)\tilde{k}),\ \text{and}\ z(\tilde{d})=z(\tilde{d}+b(b-a)\tilde{k}).\end{equation}
		\item There exists a strictly increasing sequence $\{n_k\}_{k\ge 1}\subset \N$ such that $$\sigma^{n_k}(x,x)\to (y,z)$$ as $k\to\infty$.
	\end{itemize}
	
	 Then there exist two sets $B,C\subset \N$ with $|B|=r,|C|=\infty$ and $m\in \{1,2\}$ such that
\begin{equation}\label{main goal}
\{n\in \N:x(n)=m\}\supset \big((B+aC)\cup (B+bC)\big).\end{equation}
\end{prop}
\begin{rem}
	Although at this case, the dynamical behavior of the coloring $x$ is simple, to search for specific $B$ and $C$, we have to classify all possible structures of $x$. This causes a tedious and complicated proof.
\end{rem}
\begin{proof}
	Let $\tilde{x}=T^{\tilde{d}}x$. Without loss of generality, we assume that $y(\tilde{d})=1,z(\tilde{d})=2$. Note that there exists a strictly increasing sequence $\{n_i\}_{i\in \N}$ such that $(T^{an_i}\tilde{x},T^{bn_i}\tilde{x})\rightarrow (T^{\tilde{d}}y,T^{\tilde{d}}z)$.
By \eqref{value of (y,z)} and choosing a subsequence of $\{n_i\}_{i\in \N}$,  we can find a strictly increasing seuqence $\{d_k\}_{k\ge 1}\subset \N$\footnote{Thoughout this proof, all choices of all parameters are relatively rough. But it is enough to search for targeted structures.} such that the following hold:
	\begin{itemize}
		\item For each $k\in \N$, $bd_{k}-ad_{k}-3(ab)^{2}(b-a)k>0$. And $bd_{k}-ad_{k}-3(ab)^{2}(b-a)k\to \infty$ as $k\to\infty$.
		\item For each $k\in \N$, $ad_{k+1}-bd_{k}-6(ab)^{2}(b-a)k>0$. And $ad_{k+1}-bd_{k}-6(ab)^{2}(b-a)k\to\infty$ as $k\to\infty$.
		\item For each $k\in \N,1\le i\le ak,1\le j\le 4bk$, $\tilde{x}(ad_{k}+a(b-a)i)=1$ and $\tilde{x}(bd_{k}+b(b-a)j)=2$.
          \item $(\log_{k}d_k)/k\to \infty$ as $k\to\infty$.
	\end{itemize}
	We define an arithmetic function $\phi:\N\to \N\cup \{0\} $ by the following:
	
	If there exists $ad_{k}+2ab(b-a)k<u<bd_k-a(b-a)$ such that $a(b-a)|(u-ad_{k})$ and $\tilde{x}(u+a(b-a))=2$, then
	\begin{align*}
	& k\mapsto\max\{M\in \N:\text{there exists}\ ad_{k}+2ab(b-a)k< u<bd_{k}\ \text{such that}\ a(b-a)|(u-ad_{k}),
	\\ & \hspace{2cm} \tilde{x}(u+a(b-a))=\cdots=\tilde{x}(u+Ma(b-a))=2\ \text{and}\ u+Ma(b-a)<bd_k\};
	\end{align*} otherwise, $k\mapsto 0$.
	
	Now, the rest of the proof is divided into two steps based on the property of $\phi$.
	
	\noindent \textbf{Step I. $\phi$ is bounded.}
	
	There exists $M_0\in \N$ such that for any $k\in \N$, we have $\phi(k)\le M_0$. By the definition of $\phi$, for any sufficiently large $k$ and any  $ad_k+2ab(b-a)k<u<bd_k-(M_0+1)a(b-a)$ with $a(b-a)|(u-ad_k) $, there is $1\leq n\leq M_0+1$ such that $\tilde{x}(u+na(b-a))=1$. For any $k\in \N$, let $m_k=(d_k\ \text{mod}\ a)$. Then for any $l\in \N$ and any sufficiently large $k$, by going through $u=bd_k-m_k(b-a)-n(b-a) $ from $n=1$ to $n=l(M_0+1)$, one can find at least $l$ elements
\begin{equation}\label{u sequence}\gamma_{1,k}(b-a)<\cdots<\gamma_{l,k}(b-a)\ \text{from}\ \{a(b-a),2a(b-a),\ldots,L(M_0+1)a(b-a)\}
\end{equation}
such that
\begin{equation}\label{x value 1}\tilde{x}(bd_k-m_k(b-a)-\gamma_{1,k}(b-a))=\cdots=\tilde{x}(bd_k-m_k(b-a)-\gamma_{l,k}(b-a))=1.
\end{equation}
 Choose $L$ from $\N$ such that
	$$\frac{1}{2a(M_0+1)}> (\log \log 2aL (M_0+1))^{-2^{-2^{r+9}}}$$ and fix it.
	As $L$ and $M_0$ are finite, by Lemma \ref{lem7} and pigeonhole principle, there exist a strictly increasing sequence $\{h_i\}_{i\geq 1}\subset \N$ and $m,\gamma_1,\ldots,\gamma_L$ such that $h_1>8arLM_0$ and for each $i\ge 1$, $d_{h_i}>24(ab)^{2}LM_0h_i$, $m=m_{h_i}$ and $\gamma_j=\gamma_{j,h_i},j=1,\ldots,L$.  By the choice of $L$, Theorem \ref{CT1} and \eqref{u sequence}, there exists an arithmetic progression
$$ \{s,s+t,\ldots,s+(r-1)t\}\subset \{m+\gamma_1,\ldots,m+\gamma_L\}\subset \{m+a,\ldots,m+L(M_0+1)a\}.$$
By \eqref{x value 1}, for each $i\ge 1$, we have
$$\tilde{x}(bd_{h_i}-s(b-a))=\cdots=\tilde{x}(bd_{h_i}-(s+(r-1)t)(b-a))=1. $$
Let $$C=\Big\{d_{h_i}-s-(r-1)t-a:i\in\N\Big\}\ \text{and}$$
	$$B=\Big\{\tilde{d}+a\Big(s+(r-1)t+b\Big)+j(b-a)t:0\le j\le r-1\Big\}.$$
By the restriction for $\{h_i\}_{i\ge 1}$, we have $C\subset \N$. Clearly,
$$B+aC-\tilde{d}=\Big\{ad_{h_i}+a(b-a)\Big(1+\frac{jt}{a}\Big):i\in \N  \text{ and }0 \le j\le r-1\Big\}. $$
Note that $a|((s+t)-s)$ i.e. $a|t$. By the construction of $\{d_k\}_{k\ge 1}$, one has that $\tilde{x}(ad_{k}+a(b-a)i)=1$ for each $k\in \N,1\le i\le ak$. By the restriction for $h_1$, $(1+jt/a)<ah_i$ for each $i\ge 1$. Therefore, $x(B+aC)=\tilde{x}(B+aC-\tilde{d})=1$. We also have
$$B+bC-\tilde{d}=\Big\{bd_{h_i}-(s+jt)(b-a): i\in \N  \text{ and } 0\le i\le r-1\Big\}. $$
Thus $x(B+bC)=\tilde{x}(B+bC-\tilde{d})=1$.

	\noindent \textbf{Step II. $\phi$ is unbounded.}
	
	Choose $L_0$ from $\N$ such that $L_0>ab(b-a)100^{r(r+2)}$.
	In this step, we can assume that there is a strictly increasing sequence $\{k_i\}_{i\in \N}\subset \N$ with $k_1>2ar$ such that for each $i\ge 1$, $\phi(k_i)\ge bL_0$. Define $\Lambda:\N\to\N\cup \{0\}$ by
	\begin{align*}
		& i\mapsto\min\{v-ad_{k_i}-2ab(b-a)k_i:\ ad_{k_i}+2ab(b-a)k_i\le v, a(b-a)|(v-ad_{k_i}),
		\\ & \hspace{1cm} \tilde{x}(v+a(b-a))=\cdots=\tilde{x}(v+bL_0a(b-a))=2\ \text{and}\ v+bL_0a(b-a)<bd_{k_i}\}.
	\end{align*}	
	The rest of the step is divided into two parts based on the property of $\Lambda$.
For each $i\ge 1$, let \begin{equation}\label{de of v}v_{i}=ad_{k_i}+2ab(b-a)k_i+\Lambda(i).
 \end{equation}By the definition of $\Lambda$, for each $i\in\N$,  $b(b-a)|(\frac{b}{a}v_{i}-bd_{k_i})$. By the definition of $\Lambda$, we have that for any $i\in \N$,
\begin{equation}\label{x value 2}\tilde{x}(v_i+ab(b-a))=\cdots=\tilde{x}(v_{i}+L_{0}ab(b-a))=2.\end{equation}

\noindent \textbf{Case I. $\Lambda$ is bounded.}

In this case, one has that $\{\frac{b}{a}\Lambda(i):i\in \N\}=\{\frac{b}{a}v_{i}-bd_{k_i}-2b^{2}(b-a)k_i:i\in \N\}$ is bounded. By pigeonhole principle, there exist $E\in \N\cup \{0\}$ with $b(b-a)|E$ and an infinite subset $I_1$ of $\N$ such that for each $i\in I_1$,
$$E=\frac{b}{a}v_{i}-bd_{k_i}-2b^{2}(b-a)k_i\text{ and }\frac{v_{i}}{a}>6abr(E+1).$$
	Let $$B=\Big\{\tilde{d}+a^{2}b(r+1)+\frac{aE}{b-a}+jab(b-a):1\le j\le r\Big\}\ \text{and}$$ $$C=\Big\{\frac{v_{i}}{a}-ab(r+1)-\frac{E}{b-a}:i\in I_1\Big\}.$$
By the restriction for $\{v_i\}_{i\in I_1}$, we have $C\subset \N$. Clearly,
$$B+aC-\tilde{d}=\{v_{i}+jab(b-a):1\leq j\leq r\}.$$
By \eqref{x value 2}, one has that $x(B+aC)=\tilde{x}(B+aC-\tilde{d})=2$.
We also have
$$B+bC-\tilde{d}=\{bd_{k_i}+(2bk_i+ja)b(b-a):i\in I_1,1\leq j\leq r\}.$$ Recall the definition of $\{d_k\}_{k\ge 1}$, one has that for each $i\ge1$, $\tilde{x}(bd_{k_i}+jb(b-a))=2 $ for $1\leq j\leq 4bk_i$. By the restriction for $k_1$, $ja<2bk_i$ for each $i\in I_1$.
So,
$$x(B+bC)=\tilde{x}(B+bC-\tilde{d})=2.$$
	\noindent \textbf{Case II. $\Lambda$ is unbounded.}
	
	 In this case, by the unboundedness of $\Lambda$,  there exists an infinite subset $I_2$ of $\N$ such that for $i\in I_2$, $\Lambda(i)>3aR(b-a)(bL_0+1)$, where $R\in \N$ and  $$\frac{1}{a(b-a)(bL_0+1)}> (\log \log aR(b-a)(bL_0+1))^{-2^{-2^{2ab(b-a)rL_0+10}}}.$$
By the definition of $\Lambda$, for any $i\in I_2$, there are at least $R$ elements $u_{1,i},\ldots,u_{R,i}$ in $$\{a(b-a),\ldots,aR(b-a)(bL_0+1)\}$$ such that
$$\tilde{x}(v_{i}-u_{1,i})=\cdots=\tilde{x}(v_{i}-u_{R,i})=1.$$	
By Lemma \ref{lem7} and prigeonhole principle, there exists an infinte set $I_3\subset I_2$ such that for any $i\in I_3$, $(u_{1,i},\ldots,u_{R,i})=(u_1,\ldots,u_R)$, where $\{u_1,\ldots,u_R\}\subset \{a(b-a),\ldots,aR(b-a)(bL_0+1)\}$. Here, we also choose $I_3$ such that for any $i\in I_3$, $$\frac{b}{a}v_{i}-bd_{k_i}-2b^{2}(b-a){k_i}>6(\eta+r)ab^{2}(b-a)L_0\beta,$$ where  $\eta=100a^2b^4RL_0$.	By the unboundedness of $\Lambda$ and \eqref{de of v}, this is possible. By the choice of $R$ and Theorem \ref{CT1}, there exist an arithmetic progression
$$\{\alpha,\alpha+\beta,\ldots,\alpha+2ab(b-a)rL_0\beta\}\subset \{u_1,\ldots,u_R\}\subset\{a(b-a),\ldots,aR(b-a)(bL_0+1)\}.$$
Thus one has that for any $i\in I_3$,
\begin{equation}\label{alpha value 1}\tilde{x}(v_{i}-\alpha)=\cdots=\tilde{x}(v_{i}-(\alpha+2ab(b-a)rL_0\beta))=1.
\end{equation}
	For each $i\in I_3,\eta< j\le \eta+r$, we define
\begin{equation}\label{de of Sj}
S_j(i)=\{\frac{b}{a}v_{i}-sab(b-a):(j-1)\beta L_0+1\leq s\leq (j-1)\beta L_0+L_0\}.
\end{equation}
	
	The rest of the case is divided into two situations based on the property of the sequence $\{S_{j}(i)\}_{i\in I_3,\eta<j\le \eta+r}$ of finite subsets of $\N$.
	
	\noindent\textbf{Situation I. There exist $\eta< j\le \eta+r$ and an infinite set $I_4\subset I_3$ such that for any $i\in I_4$, $|\{n\in S_j(i):\tilde{x}(n)=2\}|\ge r+2$.}

In this situation, by Lemma \ref{lem7} and pigeonhole principle, there exist an infinite set $I_5\subset I_4$ and a set $\{s_1,\ldots,s_{r}\}\subset\{1,\ldots,L_0-1\}$ such that
\begin{equation}\label{s value 2}\tilde{x}\big(\frac{b}{a}v_i-((j-1)\beta L_0+s_t)ab(b-a)\big)=2\text{ for any}\ i\in I_5,1\leq t\leq r.
\end{equation}

Let	$$B=\{\tilde{d}+((j-1)\beta +1)L_{0}a^{2}b+(L_{0}-s_{t})ab(b-a):1\le t\le r\}\ \text{and}$$ $$C=\Big\{\frac{v_{i}}{a}-((j-1)\beta +1)L_{0}ab:i\in I_5\Big\}.$$
By the restriction for $\{v_i\}_{i\in I_3}$, we have $C\subset \N$. Clearly,
$$B+aC-\tilde{d}=\{v_i+(L_{0}-s_{t})ab(b-a):1\leq t\leq r,i\in I_5\}.$$
By \eqref{x value 2}, one has that $\tilde{x}(v_i+nab(b-a))=2 $ for any $i\in I_5,1\leq n\leq L_0$. Therefore, $x(B+aC)=\tilde{x}(B+aC-\tilde{d})=2$. We also have
$$B+bC-\tilde{d}=\{\frac{b}{a}v_i-((j-1)\beta L_0+s_t)ab(b-a):i\in I_5,1\leq t\leq r\}.$$
By \eqref{s value 2}, one has that $x(B+aC)=x(B+bC)=\tilde{x}(B+bC-\tilde{d})=2$.
	
	\noindent \textbf{Situation II. Reverse side of Situation I.}
	
	In this situation, we know that there exists a co-finite subset $I_6$ of $I_3$ such that for each $i\in I_6,\eta< j\le \eta+r$, $|\{n\in S_j(i):\tilde{x}(n)=2\}|< r+2$. By \eqref{de of Sj},
$$S_j(i)=\frac{b}{a}v_i-(j-1)\beta L_0 ab(b-a)-\{ab(b-a),2ab(b-a),\ldots,L_0ab(b-a)\}. $$
By Lemma \ref{lem7} and pigeonhole principle, there exist an infinite set $I_7\subset I_6$ and a set $$A\subset \{ab(b-a),2ab(b-a),\ldots,L_0ab(b-a)\}$$
 such that for any $i\in I_7, \eta<  j\leq  \eta+r$, we have that the following hold:
 \begin{enumerate}
   \item $v_i>6\eta \beta L_0a^2b^4(b-a)$;
   \item $|A|\geq  L_0-r(r+2)>0$;
   \item for any $n\in A$, $\tilde{x}(\frac{b}{a}v_i-n-(j-1)\beta L_0 ab(b-a))=1$.
 \end{enumerate}
 Let $\xi\in A+\eta \beta L_0 ab(b-a)$. Then one has that for any $i\in I_7$,
\begin{equation}\label{xi value 1}\tilde{x}(\frac{b}{a}v_i-\xi)=\tilde{x}(\frac{b}{a}v_i-(\xi+ab(b-a)\beta  L_0))=\cdots=\tilde{x}(\frac{b}{a}v_i-(\xi+(r-1)ab(b-a)\beta L_0))=1.
\end{equation}
	Let $$B=\Big\{\tilde{d}+\frac{a(\xi-\alpha)}{b-a}-\alpha-jab(b-a)L_{0}\beta:0\le j<r-1\Big\}\ \text{and}$$ $$C=\Big\{\frac{v_{i}}{a}-\frac{\xi-\alpha}{b-a}:i\in I_7\Big\}.$$
By those restrictions for $\{v_i\}_{i\in I_7}$ and $\eta$, we have $\N\supset B,C$. Clearly,
$$B+aC-\tilde{d}=\{v_{i}-\alpha-jab(b-a)L_0\beta:i\in I_7,0\le  j< r\}.$$
By \eqref{alpha value 1}, one has that $x(B+aC)=\tilde{x}(B+aC-\tilde{d})=1$. We also have
$$B+bC-\tilde{d}=\{\frac{b}{a}v_i-\xi-jab(b-a)\beta  L_0:i\in I_7,0\leq j\leq r-1\}.$$
By \eqref{xi value 1}, one has that $x(B+bC)=\tilde{x}(B+bC-\tilde{d})=1$.

The proof is complete.
\end{proof}

\subsection{}
In the subsection, we provide two technique lemmas.
%
%
The first one is standard in measure theory. Here, we still provide a proof for completeness.
\begin{lemma}\label{lem2}
	Let $\pi:X\to Y$ be a measure preserving transformation between Lebesgue probability spaces $(X,\X,\mu)$ and $(Y,\Y,\nu)$. Let $D\in \X$ with $\mu(D)>0$. Then for any $\mu(D)>\epsilon>0$, there exists a measurable subset $D_{\ep}\subset D$ such that the following hold:
	\begin{itemize}
		\item For any $x\in D_{\ep}$, $\E_{\mu}(1_{D}|Y)(x)\ge \ep$.
		\item $\mu(D)-\mu(D_{\ep})\le\ep$.
	\end{itemize}
\end{lemma}
\begin{proof}
	Let $$\mu=\int_{Y}\mu_{y}d\nu(y)$$ be the disintegration of $\mu$ with respect to $Y$. Then for $\mu$-a.e. $x\in X$, $\E_{\mu}(1_D|Y)(x)=\mu_{\pi(x)}(D)$. Let $D'=\{x\in X:\E_{\mu}(1_D|Y)(x)\le \ep\}\cap D$.

	If $\mu(D')>\ep$, then there exists a measurable subset $D''$ with $\mu$-positive measure of $D'$ such that for $\mu$-a.e. $x\in D''$, $\mu_{\pi(x)}(D')> \ep$. It is a contradition. So, $\mu(D')\le\ep$. Let $D_{\ep}=D\backslash {D'}$. This finishes the proof.
\end{proof}
The following lemma constructs a joining.
\begin{lemma}\label{lem3}
	Let $(X_1,\B_1,\mu_1,T_1)$ and $(X_2,\B_2,\mu_2,T_2)$ be two ergodic measure preserving systems. Let $\lambda$ be an ergodic joining of $(X_1,\B_1,\mu_1,T_1)$ and $(X_2,\B_2,\mu_2,T_2)$. Let $(Z,\B(Z),m,R)$ be the Kronecker factor of $(X_1,\B_1,\mu_1,T_1)$ and $\pi$ be the related factor map from $X_1$ to $Z$. Then there exists $\hat{\lambda}\in \mathcal{M}(Z\times X_1\times X_2,\B(Z)\otimes B_1\otimes B_2)$ such that the following hold:
	\begin{itemize}
		\item $\hat{\lambda}$ is a joining of $(X_1\times X_2,\B_1\otimes \B_2,\lambda,T_1\times T_2)$ and $(Z,\B(Z),m,R)$;
		\item Let $\hat{\mu}_1$ be the projections of $\hat{\lambda}$ with respect to $Z\times X_1$. Then $(X_1,\B_1,\mu_1,T_1)$ and $(Z\times X_1,\B(Z)\otimes \B_1,\hat{\mu_1},R\times T_1)$ are isomorphic by the map $X_1\to Z\times X_1,x\mapsto (\pi(x),x)$.
	\end{itemize}
\end{lemma}
\begin{proof}
   	By \cite[Theorem 6.8]{G-book}, there is a common factor $(H,\mathcal{D},\nu,S)$ of $(X_1,\B_1,\mu_1,T_1)$ and $(X_2,\B_2,\mu_2,T_2)$ such that $$\lambda=\int_{H}\mu_{1,h}\times \mu_{2,h}d\nu(h),$$ where $$\mu_1=\int_{H}\mu_{1,h}d\nu(h)\ \text{and}\ \mu_2=\int_{H}\mu_{2,h}d\nu(h)$$ are disintegrations of $\mu_1$ and $\mu_2$ with respect to $H$, respectively. Let
   $$\hat{\lambda}=\int_{H}\Big(\int_{X_1}\delta_{\pi(x)}\times \delta_{x}d\mu_{1,h}(x)\Big)\times \mu_{2,h}d\nu(h).$$
    Clearly, $\hat{\lambda}$ is $R\times T_1\times T_2$-invariant. Note that the image of $\mu_1$ under the map $\pi$ is $m$. So, $\hat{\lambda}$ is a joining of $(X_1\times X_2,\B_1\otimes \B_2,\lambda,T_1\times T_2)$ and $(Z,\B(Z),m,R)$. Let $\hat{\mu}_1$ be the projections of $\hat{\lambda}$ with respect to $Z\times X_1$. Then $\hat{\mu}_1=\int_{X_1}\delta_{\pi(x)}\times \delta_{x}d\mu_{1}(x)$. Therefore, $\hat{\mu}_1$ is the image of $\mu_1$ under the map $X_1\to Z\times X_1,x\mapsto (\pi(x),x)$ and ${\mu}_1$ is the image of $\hat{\mu}_1$ under the map $Z\times X_1\to X_1,(x,z)\mapsto x$. This implies that  $(X_1,\B_1,\mu_1,T_1)$ and $(Z\times X_1,\B(Z)\otimes \B_1,\hat{\mu_1},R\times T_1)$ are isomorphic. This finishes the proof.
\end{proof}
\subsection{Proof of Theorem \ref{T3-1}}
\begin{proof}[Proof of Theorem \ref{T3-1}]
	Let $x\in X,r\in \N$ and fix them. Let
	\begin{align*}
		& \Lambda(x)=\text{Min}((x,x),\sigma)\cap \{(y,z)\in X\times X:\text{there}\ \text{exists}\ \text{a}\ \text{strictly}\\ & \hspace{2cm}\ \text{increasing}\ \text{sequence}
		 \ \{n_i\}_{i\ge 1}\subset \N\ \text{such}\ \text{that}\ \sigma^{n_i}(x,x)\to (y,z)\ \text{as}\ i\to\infty\}.
	\end{align*}
	By applying Lemma \ref{lem4} to $(\overline{\text{orb}((x,x),\sigma)},\sigma)$ and $(x,x)$, there exists $(y,z)\in \text{Min}((x,x),\sigma)$ and a sequence $\{n_i\}_{i\ge 1}\subset \N$ such that $ \sigma^{n_i}(x,x)\to (y,z)$ as $i\to\infty$. If $\{n_i\}_{i\ge 1}$ is unbounded, $\Lambda(x)\neq \varnothing$. Otherwise, there exists $m\in\N$ such that $\sigma^{m}(x,x)$ is a minimal point of $(X\times X,\sigma)$. By Lemma \ref{lem5}, there is a strictly increasing sequence $\{k_i\}_{i\ge 1}\subset \N$ such that $ \sigma^{k_i+m}(x,x)\to \sigma^{m}(x,x)$ as $i\to\infty$. Then $\Lambda(x)\neq \varnothing$.
	
	To sum up, $\Lambda(x)\neq \varnothing$. By Lemma \ref{lem1}, the rest proof can be divided into two cases.
	
	\noindent \textbf{Case I. There exists $(y,z)\in \Lambda(x)$ with the following: There exists $d\in \N$ such that for any $k\in \N$,
		$y(d)\neq z(d),y(d)=y(d+a(b-a)k),\ \text{and}\ z(d)=z(d+b(b-a)k)$.}
	
	Clearly, for the case, Theorem \ref{T3-1} is a direct result of Proposition \ref{pro1}.
	
	\noindent \textbf{Case II. There exist $\gamma\in\N$ and $(y,z)\in \Lambda(T^{\gamma}x)$ such that $Y\cap U\neq \varnothing$, where $Y=\overline{\text{orb}((y,z),\sigma)}$.}
	
	Recall that $U=(E_1\times E_1)\cup (E_2\times E_2)$. Without loss of generality, we can assume that $Y\cap (E_1\times E_1)\neq \varnothing$. Let $\lambda\in \mathcal{M}^{e}(Y,\sigma)$. As $(Y,\sigma)$ is minimal, we know that $$\lambda(Y\cap (E_1\times E_1))=\delta_0$$
	for some $\delta_0>0$. Let $(y_1,y_2)\in Y$ be a generic point of $\lambda$. Let $X_1=\overline{\text{orb}(y_1,T^a)}$ and $X_2=\overline{\text{orb}(y_2,T^b)}$. We can view $\lambda$ as a Borel probability measure on $X_1\times X_2$. Let $\mu_1,\mu_2$ be the projections of $\lambda$ with respect to $X_1$ and $X_2$, respectively. Let $(Z,\B(Z),m,R)$ be the Kronecker factor of $(X_1,\B(X_1),\mu_1,T^a)$ and $\pi$ be the related factor map from $X_1$ to $Z$. By Lemma \ref{lem3}, there exists $\hat{\lambda}\in \mathcal{M}(Z\times X_1\times X_2,R\times \sigma)$ such that the following hold:
	\begin{itemize}
		\item $\hat{\lambda}$ is a joining of $(X_1\times X_2,\B(X_1)\otimes \B(X_2),\lambda,T^a\times T^b)$ and $(Z,\B(Z),m,R)$;
		\item Let $\hat{\mu}_1$ be the projections of $\hat{\lambda}$ with respect to $Z\times X_1$. Then $(X_1,\B(X_1),\mu_1,T^a)$ and $(Z\times X_1,\B(Z)\otimes \B(X_1),\hat{\mu}_1,R\times T^a)$ are isomorphic by the map $X_1\to Z\times X_1,s\mapsto (\pi(s),s)$.
	\end{itemize}
	Since $X_1\cap E_1$ is clopen on $X_1$, $1_{X_1\cap E_1}$ is continuous on $X_1$. By the setting of $(y_1,y_2)$, then
	\begin{align*}
	 & \hat{\mu}_1(Z\times (X_1\cap E_1))=\mu_{1}(X_1\cap E_1)=\lim_{N\to\infty}\frac{1}{N}\sum_{n=1}^{N}1_{X_1\cap E_1}(T^{an}y_1)
	 \\ = &
	 \lim_{N\to\infty}\frac{1}{N}\sum_{n=1}^{N}1_{Y\cap \big((X_1\cap E_1)\times X_2\big)}(T^{an}y_1,T^{bn}y_2)
	 \\ \ge &
	 \lim_{N\to\infty}\frac{1}{N}\sum_{n=1}^{N}1_{Y\cap (E_1\times E_1)}(T^{an}y_1,T^{bn}y_2)
	 \\ = &
	 \lambda(Y\cap (E_1\times E_1))
	 \\ = &
	 \delta_0.
	\end{align*} The third equality follows from the fact that $1_{Y\cap (E_1\times E_1)}$ is  continuous on $Y$.

	Let $\phi=\E_{\hat{\mu}_1}(1_{Z\times (X_1\cap E_1)}|Z)$. Apply Lemma \ref{lem2} to $Z\times X_1\to Z,(z,y)\mapsto z$, $Z\times (X_1\cap E_1)$ and $\ep=\delta_0/2$. Then there exists a measurable subset $D_{\ep}\subset Z\times (X_1\cap E_1)$ such that the following hold:
		\begin{itemize}
			\item For any $(u,v)\in D_{\ep}$, $\phi(u,v)\ge \ep$;
			\item $\hat{\mu}_1(D_{\ep})\ge \hat{\mu}_1(Z\times (X_1\cap E_1))-\ep$.
		\end{itemize} As $\phi\geq 0$, we have
		\begin{align*}
			& \int\phi(u,v)1_{X_2\cap E_1}(w)d\hat{\lambda}(u,v,w)
			\\ \ge & \int_{D_{\ep}\times X_2}\phi(u,v)1_{X_2\cap E_1}(w)d\hat{\lambda}(u,v,w)
			\\ \ge &
			\frac{\delta_0}{2}\hat{\lambda}(D_{\ep}\times (X_2\cap E_1)).
		\end{align*}
		Note that
		\begin{align*}
		&  \Big|\hat{\lambda}(D_{\ep}\times (X_2\cap E_1))-{\lambda}(Y\cap (E_1\times E_1))\Big|
			\\ = &
			\Big|\hat{\lambda}(D_{\ep}\times (X_2\cap E_1))-\hat{\lambda}(Z\times (X_1\cap E_1)\times (X_2\cap E_1))\Big|
			\\ \le &
			\hat{\mu}_1(Z\times (X_1\cap E_1))-\hat{\mu}_1(D_{\ep})
			\\ \le &
			\frac{\delta_0}{2}.
		\end{align*}
		So, we have that
		\begin{equation}\label{eq3}
			\int\phi(u,v)1_{X_2\cap E_1}(w)d\hat{\lambda}(u,v,w)\ge \frac{\delta_0}{2}({\lambda}(Y\cap (E_1\times E_1))-\frac{\delta_0}{2})=\frac{\delta_0^2}{4}.
		\end{equation}
		Let $\delta=\delta_{0}^{2}/16$. Since $Z$ is a compact abelian group, by the related Character Theory (see \cite{HR}), there exists an open subset $V$ containing identity element $e_{Z}$ of $Z$ such that for any $\xi\in V$,
		\begin{equation}\label{eq1}
			\int_{Z}|\phi(u+\xi)-\phi(u)|^{2}dm(u)<\delta^2/16,
		\end{equation}
		 where we view $\phi$ as an element of $L^{2}(Z,\B(Z),m)$.
		
		Let $\psi=1_{Z\times (X_1\cap E_1)}-\phi$ and $\theta\in L^{2}(\hat{\mu}_1)$\footnote{To get the specific form of $\theta$, one can see \cite[Page 129]{G-book}.} such that for any $g\in L^{2}(\hat{\mu}_1)$, $$\int 1_{X_2\cap E_1}\cdot g d\hat{\lambda}=\int \theta \cdot gd\hat{\mu}_1.$$
		By Birkhoff's ergodic theorem, we have
		\begin{align*}
			&\lim_{N\to\infty}\frac{1}{N}\sum_{n=1}^{N}\Big|\int \theta\cdot(R\times T^a)^{(b-a)n}\psi d\hat{\mu}_1\Big|^{2}
			\\ = &
			\lim_{N\to\infty}\frac{1}{N}\sum_{n=1}^{N}\int ((R\times T^a)\times (R\times T^a))^{(b-a)n}(\psi \otimes \bar{\psi})\cdot (\theta \otimes \bar{\theta})d\hat{\mu}_1\times \hat{\mu}_1
			\\ = &
			\int \Big(\lim_{N\to\infty}\frac{1}{N}\sum_{n=1}^{N} ((R\times T^a)\times (R\times T^a))^{(b-a)n}(\psi \otimes \bar{\psi})\Big)\cdot (\theta \otimes \bar{\theta})d\hat{\mu}_1\times \hat{\mu}_1
			\\ = &
			\int \E_{\hat{\mu}_1\times \hat{\mu}_1}(\psi \otimes \bar{\psi}|\mathcal{I}((R\times T^a)\times (R\times T^a)))\cdot (\theta \otimes \bar{\theta})d\hat{\mu}_1\times \hat{\mu}_1,
		\end{align*} where $\mathcal{I}((R\times T^a)\times (R\times T^a))$ is a sub-$\sigma$-algebra, spanned by all $(R\times T^a)\times (R\times T^a)$-invariant functions, on $(Z\times X_1)\times (Z\times X_1)$.
		Note that $(X_1,\B(X_1),\mu_1,T^a)$ and $(Z\times X_1,\B(Z)\otimes \B(X_1),\hat{\mu}_1,R\times T^a)$ are isomorphic. Then by Lemma \ref{lem6}, $(Z,\B(Z),m,R^{b-a})$ is the Kronecker factor of $(Z\times X_1,\B(Z)\otimes \B(X_1),\hat{\mu}_1,(R\times T^a)^{b-a})$. By Lemma \ref{eigenfunctions and invariant}, Corollary \ref{orthogonal with eigenfunctions} and $ \E_{\hat{\mu}_1}(\psi|Z)=0$, one has that  $$\lim_{N\to\infty}\frac{1}{N}\sum_{n=1}^{N}\Big|\int \theta\cdot(R\times T^a)^{(b-a)n}\psi d\hat{\mu}_1\Big|^{2}=0.$$
		So, $$\lim_{N\to\infty}\frac{|\{n\in B_0:n\le N\}|}{N}=0,$$ where
		$$B_0=\Big\{n\in\N:\Big|\int \theta\cdot(R\times T^a)^{(b-a)n}\psi d\hat{\mu}_1\Big|\ge \delta/4\Big\}.$$

Let $$B_1=\{n\in \N:R^{(b-a)n}e_{Z}\in V\}\backslash B_0.$$ By \cite[Proposition 4.5]{HK}, $(\overline{\{R^{(b-a)n}e_{Z}:n\in \Z\}},R^{b-a})$ is a minimal topological dynamical system. Since $V\cap \overline{\{R^{(b-a)n}e_{Z}:n\in \Z\}}\neq \varnothing$, by Lemma \ref{lem5}, we know that $$\limsup_{N\to\infty}\frac{|\{n\in \N:n\le N,R^{(b-a)n}e_{Z}\in V\}|}{N}>0.$$ Hence, $B_1$ has positive upper Banach density. Note that $\norm{\theta}_{L^{2}(\hat{\mu}_1)}\le 1$. By \eqref{eq1} and the definition of $B_1$, for any $n\in B_1$, 
		\begin{equation}\label{eq2}
			\Big|\int ((R\times T^a)^{(b-a)n}\phi-\phi)\cdot \theta d\hat{\mu}_1\Big|< \delta/4\ \text{and}\ \Big|\int \theta\cdot(R\times T^a)^{(b-a)n}\psi d\hat{\mu}_1\Big|<\delta/4.
		\end{equation}
		By \eqref{eq2}, the definitions of $\psi$ and $\theta$, \eqref{eq3} and $\ep^{2}=4\delta$, we know that for any $n\in B_1$,
		\begin{align*}
			& \Big|\int 1_{X_2\cap E_1} \cdot (R\times T^a)^{(b-a)n}1_{Z\times (X_1\cap E_1)}d\hat{\lambda}\Big|
			\\ = &
			\Big|\int \theta \cdot (R\times T^a)^{(b-a)n}1_{Z\times (X_1\cap E_1)}d\hat{\mu}_1\Big|
			\\ = &
			\Big|\int \theta \cdot (R\times T^a)^{(b-a)n}\psi d\hat{\mu}_1+\int \theta \cdot ((R\times T^a)^{(b-a)n}\phi -\phi) d\hat{\mu}_1+\int \theta\phi d\hat{\mu}_1\Big|
			\\ \ge &
			\Big|\int \theta\phi d\hat{\mu}_1\Big|-\Big|\int \theta \cdot (R\times T^a)^{(b-a)n}\psi d\hat{\mu}_1\Big|-\Big|\int \theta \cdot ((R\times T^a)^{(b-a)n}\phi -\phi) d\hat{\mu}_1\Big|
			\\ \ge &
			\delta.
		\end{align*}
	That is, for any $n\in B_1$,
		\begin{equation}\label{eq4}
			\lambda(T^{-a(b-a)n}(X_1\cap E_1)\times (X_2\cap E_1))\ge \delta.
		\end{equation}
		Since $\lambda$ is $\sigma$-invariant, $$\inf_{n\in B_1}\lambda(\tau^{-ban}((X_1\cap E_1)\times (X_2\cap E_1)))\ge \delta.$$
		By Theorem \ref{CT2}, there exists $B_2\subset B_1$ with $|B_2|=r$ such that
		$$\lambda(\bigcap_{n\in B_2}\tau^{-ban}((X_1\cap E_1)\times (X_2\cap E_1)))>0.$$ As $\displaystyle 1_{Y\bigcap (\bigcap_{n\in B_2}\tau^{-ban}((X_1\cap E_1)\times (X_2\cap E_1)))}$ is a continuous function on $Y$ and $(y_1,y_2)$ is a generic point of $\lambda$, the upper Banach density of $C$ is positive, where $$C=\{n\in\N: \sigma^{n}(y_1,y_2)\in \bigcap_{n\in B_2}\tau^{-ban}((X_1\cap E_1)\times (X_2\cap E_1))\}.$$
		
		Let $B_3=baB_2$. Then for any $h\in B_3,k\in C$, we have $$\tau^{h}\sigma^{k}(y_1,y_2)\in E_1\times E_1.$$ Let $A=\bigcap_{n\in B_3}\tau^{-n}(E_1\times E_1)$. Since $B_3$ is finite, $A$ is open on $X\times X$.
		
		We write $C$ as $\{k_1<k_2<\cdots\}$. By the setting of $(y,z)$,
		there exists $l_1\in \Z$ such that $\sigma^{l_1+k_1}(y,z)\in A$. By the setting of $(y,z)$, one can find $s_1\in \N$ such that $s_1+l_1+k_1\in \N$ and $\sigma^{s_1+l_1+k_1}\tau^{\gamma}(x,x)\in A$. Repeat the process. Finally, we can find
		  a strictly increasing sequence $\{t_i\}_{i\ge 1}\subset \N$ such that $$\sigma^{t_i}\tau^{\gamma}(x,x)\in A.$$
		
		Now, let $\tilde{B}=\gamma+B_3$ and $\tilde{C}=\{t_i:i\ge 1\}$. This finishes the proof.
	\end{proof}

\section{Proofs of Propositions \ref{P1}, \ref{P2}}\label{SC}
In this section, we show Propositions \ref{P1}, \ref{P2} by constructing some specific counter-examples.

\subsection{Proof of (1) of Proposition \ref{P1}}
\begin{proof}
		Now, we view $Q$ and $P$ as two elements of $\R[t]$. By the assumption on $P$ and $Q$, there exists $a>0$ such that for any $t>a$, the following hold:
	\begin{itemize}
		\item $Q(t)>P(t)\ge 1$;
		\item $Q'(t)>P'(t)\ge 1$.
	\end{itemize}
	Denote the inverse function of $P(t)$ on $(P(a),\infty)$ by $P^{-1}(t)$. Define $\psi:(P(a),\infty)\to (Q(a),\infty)$ by $t\mapsto Q\circ P^{-1}(t)$. In this case, we have that for any $t>P(a),\psi'(t)>1,\psi(t)-t>0$ and $\psi(t)-t$ is strictly increasing on $(P(a),\infty)$. Let $\delta=\deg Q/\deg P$ and $c$ be the ratio of the leading coefficients of $Q$ and $P$. Then we have
	\begin{equation}\label{eq100}
		\psi(t)= ct^{\delta}+\Phi(t),
	\end{equation}
	where $\Phi$ is a smooth function on $(P(a),\infty)$ such that
\begin{itemize}
\item ${\displaystyle\lim_{t\rightarrow \infty}\frac{\Phi(t)}{t^{\delta}}}=0$;
\item $\Phi\equiv0$ or ${\displaystyle\lim_{t\to \infty}\frac{\log t}{\Phi(t)}=0.}$
\end{itemize}
	
	The rest proof is divided into two cases.
	
	\noindent \textbf{Case I. $\delta>1$ or $\delta=1,c>1$}.
		Let $f(t)=\log t$ on $(0,\infty)$. Let $a_0$ be a sufficiently large positive integer  sufficing three properties listed below.

\noindent \textbf{Property I:}
 For any $t>a_0$ and any $r\in [0,f\circ\psi(t)),s\ge 0,r+Q(s)\ge t/2,r,s\in\N$, the equation
 $$r+Q(s)=x+Q(y)$$
  has the unique integer solution $x=r,y=s$ in the region
 $$\{(x,y)\in \R^2:0\le x<f\circ \psi(t),y\ge 0\}.$$

 The reason why this property can hold follows from the fact that $\deg Q>1$.

\noindent \textbf{Property II:} For any $t>a_0/2$ and any $0\leq h<\left\{
                                                                                                                                                      \begin{array}{ll}
                                                                                                                                                        \frac{2c-2}{3c}f(t), & \hbox{if $\delta=1$;} \\
                                                                                                                                                        \frac{1}{2}f(t), & \hbox{if $\delta>1$.}
                                                                                                                                                      \end{array}
                                                                                                               \right.$,
 \begin{equation}\label{c1-1}
 \psi(t+f(t)-h)-\psi(t)>f(\psi(t))-h.
 \end{equation}

The following statement explaines the reason why this property can hold. By using Mean Value Theorem, we have that
\begin{equation}\label{c1-2}
\psi(t+f(t)-h)-\psi(t)=(f(t)-h)\psi'(t+\theta),
\end{equation}
where $\theta\in [0,f(t)-h]$. Note that
	\begin{align*}
       & h+(f(t)-h)\psi'(t+\theta)-f(\psi(t))
		\\ = &
		\psi'(t+\theta)\Big(f(t)-h+\frac{h}{\psi'(t+\theta)}-\frac{\log c+\delta\log t+\log \big(1+\frac{\Phi(t)}{ct^{\delta}}\big)}{c\delta(t+\theta)^{\delta-1}+\Phi'(t+\theta)}\Big).
	\end{align*}
Recall two properties of $\Phi$. When $t$ is sufficiently large, we have that
$$\frac{\log c+\delta\log t+\log \big(1+\frac{\Phi(t)}{ct^{\delta}}\big)}{c\delta(t+\theta)^{\delta-1}+\Phi'(t+\theta)}\leq \left\{
                                                                                                                                                      \begin{array}{ll}
                                                                                                                                                        \frac{c+2}{3c}f(t), & \hbox{if $\delta=1$;} \\
                                                                                                                                                        \frac{1}{2}f(t), & \hbox{if $\delta>1$.}
                                                                                                                                                      \end{array}                                                                                                        \right.$$
 Thus we have
$$h+(f(t)-h)\psi'(t+\theta)-f(\psi(t))>0.$$
Together with \eqref{c1-2}, we have \eqref{c1-1}.

\noindent \textbf{Property III:}
\begin{itemize}
\item There exists $\lambda_0>\delta$ such that $\psi'(t)>\lambda_0^2$ for any $t>f(a_0)$;\footnote{If $\delta=1$, then the existence of $\lambda_0$ is from $\psi'(t)>1$. If $\delta>1$, then the existence of $\lambda_0$ is clear.}
\item Let $\epsilon_0=(\lambda_0-\delta)/2$. For any $t>a_0/2$, we have that  $f(\psi(t))<(\lambda_0-\epsilon_0)f(t)$;\footnote{It follows from \eqref{eq100} and the properties of $\Phi$.}
\item $\psi(a_0-f(a_0))-a_0-f(\psi(a_0))\geq uf(a_0)$, where $u$ is a positive constant whose specific value will be given later.\footnote{By \eqref{eq100}, $u$ always exists.}
\end{itemize}

	 Let $A_{0}=[a_0,a_0+f(a_0))\cap \N$. For each $n\in \N$, let $$a_n=\psi\circ \cdots \circ \psi(a_0)\ (n\ \text{times})\ \text{and}$$
	 \begin{align*}
	 	& A_{n}=([a_n,a_n+f(a_{n}))\cap\N)\cup \{i+Q(j):i,j\in\Z,
	 	\\ & \hspace{3cm}
	 	0\le i<f(a_{n-1}),j\ge 0,i+P(j)\in A_{n-1}\}.
	 \end{align*}
	 Let $$A=\bigcup_{n\ge 0}A_n.$$
	
	For the sequence $\{A_n\}_{n\ge 0}$, there is a claim, whose proof is displayed in the ending of the case.
	\begin{cl}\label{cl1}
		For any $n\ge 0$, $\inf A_{n+1}>\sup A_n$.
	\end{cl}
	Now, we define a $2$-coloring of $\N$, denoted by $\phi$, by the following:
	
	For any $z\in\N$, if there exists $n\in \N$ such that $z\in [a_{2n-1},a_{2n})\backslash A$, we assign color $1$ to $z$; if there exists $n\ge 0$ such that $z\in [a_{2n},a_{2n+1})\backslash A$, we assign color $2$ to $z$; if there exists $n\ge 0$ such that $z\in A_{2n}$, we assign color $2$ to $z$; if there exists $n\ge 0$ such that $z\in A_{2n+1}$, we assign color $1$ to $z$; otherwise, we assign color $1$ to $z$. The coloring is well-defined due to Claim \ref{cl1}.
	
	Take $h$ from $\N$ arbitrarily and fix it. Assume that there exists a strictly increasing sequence $\{k_i\}_{i\ge 1}\subset \N$ such that $\phi$ takes constant value on the set $$\{h+P(k_1),h+P(k_2),\ldots\}\cup \{h+Q(k_1),h+Q(k_2),\ldots\}.$$
	
	Take sufficiently large $k$\footnote{The following restrictions on $k$ are possible due to Claim \ref{cl1} and strictly increasing behavior of $\{a_m\}_{m\ge 0}$.} from $\{k_i:i\ge 1\}$ such that
\begin{itemize}
\item $k>a_1$;
\item for any $m\in \N$ with $a_{m}\geq P(k)$, we have $h<\left\{    \begin{array}{ll}
                                                                                         \frac{2c-2}{3c}f(a_{m-1}), & \hbox{if $\delta=1$;} \\
                                                                 \frac{1}{2}f(a_{m-1}), & \hbox{if $\delta>1$.}
                                                                                                                                   \end{array}
                                                                    \right.$.
\item for any $n\in \N$ with $\inf A_{n}>k $, we have $h<f(\inf A_{n-1}) $;
\item for any $m\in \N$ with $a_{m}\geq P(k)$, we have $a_{m-1}-h>a_{0}$.
\end{itemize}
Then:
	\begin{itemize}
		\item If $\phi(h+P(k))=2$:
		
		If $h+P(k)\in A$, then there exists $n\ge 0$ such that $h+P(k)\in A_{2n}$. So, by thoes restrictions for $k$, $h+Q(k)\in A_{2n+1}$. That is, $\phi(h+Q(k))=1$. It is a contradiction.
		
		If $h+P(k)\notin A$, then there exists $n\ge 0$ such that $h+P(k)\in [a_{2n},a_{2n+1})$. By $\psi'>1$ on $(a_0,\infty)$ and those restrictions for $k$, we have that
		\begin{align}
			& h+Q(k) =h+ \psi(P(k))<h + \psi(a_{2n+1}-h) \notag
			\\ = &
			\psi(a_{2n+1})+h(1-\psi'(a_{2n+1}-\tilde{h}))\hspace{0.5cm}\text{(Mean Value Theorem)} \label{eq6}
			\\ < &
			a_{2n+2}, \notag
		\end{align} where $\tilde{h}\in [0,h]$. Meanwhile, one has that $$([a_{2n},a_{2n+1})\cap \N)\setminus A\subset [a_{2n}+f(a_{2n}),a_{2n+1})\cap \N.$$ Hence, $P(k)\geq a_{2n}+ f(a_{2n})-h$. By \eqref{c1-1} and those restrictions for $k$, we have that
		\begin{align}
			& Q(k) =\psi(P(k))\geq \psi(a_{2n}+ f(a_{2n})-h) \notag
			\\ \geq &
			\psi(a_{2n})+f(\psi(a_{2n}))-h=a_{2n+1}+f(a_{2n+1})-h.
		\end{align}	
Therefore, $h+Q(k)\in [a_{2n+1}+f(a_{2n+1}),a_{2n+2})$. So, by \textbf{Property I} of $a_0$ and $h+P(k)\notin A$,  $h+Q(k)\notin A$. Then $\phi(h+Q(k))=1$. It is a contradiction.
		\item If $\phi(h+P(k))=1$:
		
		If $h+P(k)\in A$, then there exists $n\ge 0$ such that $h+P(k)\in A_{2n+1}$. So, by thoes restrictions for $k$, $h+Q(k)\in A_{2n+2}$. That is, $\phi(h+Q(k))=2$. It is a contradiction.
		
		If $h+P(k)\notin A$, then there exists $n\ge 0$ such that $h+P(k)\in [a_{2n+1},a_{2n+2})$. By some calculations like ones of the first case, we know that $h+Q(k)<a_{2n+3}$ and $h+Q(k)\ge a_{2n+2}+f(a_{2n+2})$. So, by \textbf{Property I} of $a_0$ and $h+P(k)\notin A$,  $h+Q(k)\notin A$. Then $\phi(h+Q(k))=2$. It is a contradiction.
	\end{itemize}
	
	To sum up, for any $i\in \{1,2\}$ and any $n\in \N$, the set $$\{m\in \N:\phi(n+P(m))=\phi(n+Q(m))=i\}$$ is empty or finite.
	
	Now, let us begin the proof of Claim \ref{cl1}.

		Let $b_0=a_0+f(a_0)$ and $c_0=a_0$. For any $n\in\N$, let
$b_n=\psi(b_{n-1})$ and
$c_n=\psi(c_{n-1}-f(a_{n-1}) )+f(a_{n-1})$.

Recall $\lambda_0$ and $\epsilon_0$ in \textbf{Property III} of $a_0$. Now, we show that for each $n\ge 0$, $$c_{n+1}-f(a_{n+1})-b_n \geq uf(a_n),$$
where $u=\frac{(\lambda_0-\epsilon_0)(\lambda_0-\epsilon_0-1)}{2\lambda_0\epsilon_0-\epsilon_0^2}$.
By \textbf{Property III} of $a_0$, we have $c_{1}-f(a_{1})-b_0 \geq uf(a_0)$.

Assume that when $m=n-1$, $$c_{m+1}-f(a_{m+1})-b_m \geq uf(a_m),\ \text{where}\ n\ge 1.$$ Next, we prove that $$c_{n+1}-f(a_{n+1})-b_n \geq uf(a_n).$$

By \textbf{Property III} of $a_0$, one has that $\psi'(t)>\lambda_0^2$ for any $t>f(a_0)$ and for any $n$, $f(a_{n+1})=f(\psi(a_n))<(\lambda_0-\epsilon_0)f(a_n)$. Then by Mean Value Theorem, one has that for some $t\in (b_{n-1},c_n-f(a_n))$,
\begin{align*}
& c_{n+1}-f(a_{n+1})-b_n=\psi(c_{n}-f(a_n))+f(a_n)-f(\psi(a_n))-\psi(b_{n-1})
\\ = &
\psi'(t)(c_n-f(a_n)-b_{n-1})-(f(\psi(a_n))-f(a_n))
\\ \ge &
\lambda_0^2(c_n-f(a_n)-b_{n-1})-(\lambda_0-\epsilon_0-1)f(a_n)
\\ \ge &
\lambda_0^2 uf(a_{n-1})-(\lambda_0-\epsilon_0-1)f(a_n)
\\ \ge &
\frac{\lambda_0^2u}{\lambda_0-\epsilon_0}f(a_n)-(\lambda_0-\epsilon_0-1)f(a_n)
\\ = &
(\frac{\lambda_0^2u}{\lambda_0-\epsilon_0}-\lambda_0+\epsilon_0+1)f(a_n)
\\ \ge &
u(\lambda_0-\epsilon_0)f(a_n)\geq uf(a_{n}),
\end{align*}
which completes the induction process. Note that $u>0$ and for each $n\ge 0$, $a_n>1$. Thus we have that for each $n\ge 0$, $ c_{n+1}-f(a_{n+1})-b_n>0$.

By the definition of $A_0$, we have $\inf A_0\ge a_0=c_0$. Assume that $\inf A_n\ge c_n$, where $n\ge 0$. Next, we show that $\inf A_{n+1}\ge c_{n+1}$.
As $\psi'>1$ on $(f(a_0),\infty)$, we have that for any $m\ge c_n$,
\begin{equation*}\label{decreasing}H_{m}(u)=\psi(m-u)+u \text{ is decreasing on }[0,f(a_n)).
\end{equation*} Therefore, $\inf A_{n+1}\ge \psi(c_{n}-f(a_n))+f(a_n)=c_{n+1}$, which completes the induction process.

By the definition of $A_0$, we have $\sup A_0\le b_0$. By \textbf{Property II} of $a_0$ and $\psi'>1$ on $(f(a_0),\infty)$, one has that $\sup A_1\leq b_1$. By an inductive argument, we have $\sup A_n\leq b_n$ for any $n\ge 0$.

  To sum up, for each $n\ge 0$, $\sup A_n\leq b_n<c_{n+1}\leq \inf A_{n+1}$, which proves the claim.

	\noindent \textbf{Case II. $\delta=c=1$.}
	
	There exists $l\in \N$ such that there exists $N_0\in\N$ with $P(N_0)>1$ such that for any $n\ge N_0$, $P(n)$ and $Q(n)$ are strictly increasing and $$P(n+l-1)\le Q(n)\le P(n+l).$$
	
	The rest of the case is divided into three parts.
	
	\noindent\textbf{Part I. $\lim_{n\to\infty}(Q(n)-P(n+l-1))=\infty$ and $\lim_{n\to\infty}(P(n+l)-Q(n))=\infty$.}
	
	Let $z\in\N$. If $z<P(N_0+l-1)$, we assign color $1$ to $z$; if there exists $n\ge N_0$ such that $P(n+l-1)\le z<Q(n)$, we assign color $2$ to $z$; otherwise, we assign color $1$ to $z$. Then we define a $2$-coloring on $\N$, denoted by $\phi_0$.
	
Now, we show that for any $i\in \{1,2\}$ and any $n\in \N$, the set $$\{m\in \N:\phi_0(n+P(m))=\phi_0(n+Q(m))=i\}$$denoted by $L_n$ is empty or finite.

Take $n$ from $\N$ arbitrarily and fix it. Choose $M_0 $ from $\N$ such that $ M_0>N_0+2l$ and for any $m>M_0$, $n+Q(m)\in [Q(m),P(m+l)) $ and $n+P(m)\in [P(m),Q(m-l-1)) $. By the definition of $\phi_0$, we have that $\phi_0(n+P(m))=2 $ and $\phi_0(n+Q(m))=1$ for any $m>M_0$. Thus $L_n$ has at most $M_0$ elements.
	
	\noindent \textbf{Part II. $\lim_{n\to\infty}(Q(n)-P(n+l-1))=\infty$ and $\{P(n+l)-Q(n):n\in\N\}$ is bounded.}
	
	We assign color $1$ to $\{1,\ldots,Q(N_0)-1\}$; for any $k\ge 0$, we assign color $(k\ \text{mod}\ 2)$ to $[Q(N_0+kl),Q(N_0+(k+1)l))\cap \N$. Then we define a $2$-coloring on $\N$, denoted by $\phi_1$.
		
Now we show that for any $i\in \{0,1\}$ and any $n\in \N$, the set $$\{m\in \N:\phi_1(n+P(m))=\phi_1(n+Q(m))=i\}$$denoted by $L_n$ is empty or finite.

Let $K_1=2\sup \{P(n+l)-Q(n):n\in\N\}$. Take $n$ from $\N$ arbitrarily and fix it.  Choose $M_1$ from $\N$ such that $M_1>N_0+2l$ and for any $s>M_1$, $Q(s-l+1)-Q(s-l)>n+K_1 $. Take $m$ from $(M_1,\infty)\cap \N$ arbitrarily and fix it. Then
$$n+Q(m)\in [Q(N_0+k_ml),Q(N_0+(k_m+1)l), \text{where}\ k_ml+N_0=m.$$
 And $\phi_1(n+Q(m))=(k_m\ \text{mod } 2)$. Notice that $P(m)\in [Q(m-l),Q(m-l)+K_1) $, thus one has that
\begin{equation*}\begin{array}{ll}n+P(m)\in [Q(m-l)+n,Q(m-l)+n+K_1)\\
\subset  [Q(m-l),Q(m-l+1))\subset [Q(N_0+(k_m-1)l),Q(N_0+k_ml)).
\end{array}
\end{equation*}
So, one has that $\phi_1(n+P(m))=((k_m-1)\ \text{mod } 2)\neq \phi_1(n+Q(m)) $. Thus $L_n$ has at most $M_1$ elements.

	\noindent \textbf{Part III. $\lim_{n\to\infty}(P(n+l)-Q(n))=\infty$ and $\{Q(n)-P(n+l-1):n\in\N\}$ is bounded.}
	
	At this case, $l>1$. We assign color $1$ to $(0,P(N_0))\cap \N$; for any $k\ge 0$, we assign color $(k\ \text{mod}\ 2)$ to $[P(N_0+k(l-1)),P(N_0+(k+1)(l-1)))\cap \N$. Then we define a $2$-coloring on $\N$, denoted by $\phi_2$.
	
	We will show that for any $i\in \{0,1\}$ and any $n\in \N$, the set $$\{m\in \N:\phi(n+P(m))=\phi(n+Q(m))=i\}$$denoted by $L_n$ is empty or finite.
	
	Let $K_2=2\sup \{Q(n)-P(n+l-1):n\in\Z\}$. Take $n$ from $\N$ arbitrarily and fix it.  Choose $M_2$ from $\N$ such that $M_2>N_0+2l$ and for any $s>M_2 $, $P(s-(l-1))-P(s-l)>n+K_2$. Take $m$ from $(M_2,\infty)\cap \N$ arbitrarily and fix it. Then
$$n+P(m)\in [P(N_0+s_m(l-1)),P(N_0+(s_m+1)(l-1))),\ \text{where}\ m=N_0+s_m(l-1).$$
One also has that $n+P(m+l-1)\leq n+Q(m)\leq n+P(m+l-1)+K_2$, thus one has that
$$n+Q(m)\in [P(N_0+(s_m+1)(l-1)),P(N_0+(s_m+2)(l-1)),$$
which implies that $\phi_2(n+Q(m))=((s_m+1)\ \text{mod}\ 2)\neq \phi_2(n+P(m))$. Thus $L_n$ has at most $M_2$ elements.

	The proof is complete.
\end{proof}

\subsection{Proof of (2) of Proposition \ref{P1}}
\begin{proof}
	Let $P(n)=an$ and $Q(n)=bn$, where $a,b\in \N$ and $a<b$. Take $l$ from $[\frac{b^2}{a^2},\frac{b^3}{a^3})$ arbitrarily. Then we take $x$ from $(l\cdot\frac{a}{b},\frac{b^2}{a^2})$ arbitrarily. Lastly, we take $y$ from $(x\cdot \frac{a}{b},\sqrt{x})$ arbitrarily. Clearly, $1<y<x<l$.
	
	Now, we define a $3$-coloring of $(0,\infty)$ by the following:
	
	For any $z\in (0,\infty)$, if  there exists $m\in \N$ such that $z\in [l^{m},y\cdot l^{m})$, then we assign color $1$ to $z$; if there exists $m\in \N$ such that $z\in [y\cdot l^{m},x\cdot l^{m})$, then we assign color $2$ to $z$; if there exists $m\in \N$ such that $z\in [x\cdot l^{m},l^{m+1})$, then we assign color $3$ to $z$; otherwise, we assign color $1$ to $z$.
	
	It is easy to see that $\N$ can inherite a $3$-coloring, denoted by $\phi$, from the above coloring of $(0,\infty)$.
	
	Take $n$ from $\N$ arbitrarily and fix it.  Assume that there exists a strictly increasing sequence $\{k_i\}_{i\ge 1}\subset \N$ such that $\phi$ takes constant value on the set $$\{n+ak_1,n+ak_2,\ldots\}\cup \{n+bk_1,n+bk_2,\ldots\}.$$
	
	Take sufficiently large $k$ from $\{k_i:i\ge 1\}$such that $k>l$ and for any $m$ with $l^{m+1}>ak$, we have $(\frac{b}{a}-y)l^m>(\frac{b}{a}-1)n$, $(\frac{b}{a}y-x)l^m>(\frac{b}{a}-1)n$ and $(\frac{b}{a}x-l)l^m>(\frac{b}{a}-1)n$. Then:
	\begin{itemize}
		\item If $\phi(n+ak)=1$, then there exists $m\in \N$ such that $n+ak\in [l^m,y\cdot l^m)$ and $$n+bk\in [\frac{b}{a}l^{m}+n-\frac{bn}{a},\frac{b}{a}yl^{m}+n-\frac{bn}{a}).$$
		Note that $$\frac{b}{a}yl^{m}+n-\frac{bn}{a}<\frac{b}{a}yl^{m}<l^{m+1}\ \text{and}$$
		$$\frac{b}{a}l^{m}+n-\frac{bn}{a}=yl^{m}+(\frac{b}{a}-y)l^{m}+n-\frac{bn}{a}>yl^{m}.$$ So, $\phi(n+bk)\neq 1$. It is a contradiction.
		\item If $\phi(n+ak)=2$, then there exists $m\in \N$ such that $n+ak\in [yl^m,x\cdot l^m)$ and $$n+bk\in [\frac{b}{a}yl^{m}+n-\frac{bn}{a},\frac{b}{a}xl^{m}+n-\frac{bn}{a}).$$ Note that $$\frac{b}{a}xl^{m}+n-\frac{bn}{a}<\frac{b}{a}xl^{m}<yl^{m+1}\ \text{and}$$
		$$\frac{b}{a}yl^{m}+n-\frac{bn}{a}=xl^{m}+(\frac{b}{a}y-x)l^{m}+n-\frac{bn}{a}>xl^{m}.$$ So, $\phi(n+bk)\neq 2$. It is a contradiction.
		\item If $\phi(n+ak)=3$, then there exists $m\in \N$ such that $n+ak\in [xl^m, l^{m+1})$ and $$n+bk\in [\frac{b}{a}xl^{m}+n-\frac{bn}{a},\frac{b}{a}l^{m+1}+n-\frac{bn}{a}).$$ Note that $$\frac{b}{a}l^{m+1}+n-\frac{bn}{a}<\frac{b}{a}l^{m+1}<xl^{m+1}\ \text{and}$$ $$\frac{b}{a}xl^{m}+n-\frac{bn}{a}=l^{m+1}+(\frac{b}{a}x-l)l^{m}+n-\frac{bn}{a}>l^{m+1}.$$ So, $\phi(n+bk)\neq 3$. It is a contradiction.
	\end{itemize}
	To sum up, for any $n\in \N,i\in \{1,2,3\}$, the set $$\{m\in \N:\phi(n+am)=\phi(n+bm)=i\}$$ is empty or finite. This finishes the proof.
 \end{proof}
\subsection{Proof of (3) of Proposition \ref{P1}}
\begin{proof}
	Let $P(n)=an$ and $Q(n)=bn$, where $a,b\in \N$ and $a<b$. Now, we define a $2$-coloring of $(0,\infty)$ by the following:
	
	For any $z\in (0,\infty)$, if there exists $m\in \N$ such that $z\in [(\frac{b}{a})^{2m},(\frac{b}{a})^{2m+1})$, then we assign color $1$ to $z$; if If there exists $m\in \N$ such that $z\in [(\frac{b}{a})^{2m+1},(\frac{b}{a})^{2(m+1)})$, then we assign color $2$ to $z$; otherwise, we assign color $1$ to $z$.
	
	It is easy to see that $\N$ can inherite a $2$-coloring, denoted by $\phi$, from the above coloring of $(0,\infty)$.
	
	Assume that there exist infinite $B,C\subset \N$ such that $\phi$ takes constant value on the set $(B+aC)\cup (B+bC)$.
	
	We take $n_1,n_2$ from $B$ such that $an_2>bn_1$ and take $k$ from $C$ such that $$(\frac{b}{a})^{m}\le ak<(\frac{b}{a})^{m+1},$$ for some $m>2$ with $n_1+n_2<(\frac{b}{a})^{m-1}(\frac{b}{a}-1)$. As $\phi$ takes constant value on the set $(B+aC)\cup (B+bC)$, $$\{n_1+ak,n_1+bk,n_2+ak,n_2+bk\}\subset [(\frac{b}{a})^{m+1},(\frac{b}{a})^{m+2}).$$
	So, $$\frac{(\frac{b}{a})^{m+1}-n_1}{a}\le k<\frac{(\frac{b}{a})^{m+2}-n_1}{b}\ \text{and}$$ $$\frac{(\frac{b}{a})^{m+1}-n_2}{a}\le k<\frac{(\frac{b}{a})^{m+2}-n_2}{b}.$$ By the assumption that $an_2>bn_1$, we have that $\frac{(\frac{b}{a})^{m+1}-n_1}{a}>\frac{(\frac{b}{a})^{m+2}-n_2}{b}$. It is a contradiction. This finishes the proof.
\end{proof}
\subsection{Proof of Proposition \ref{P2}}
\begin{proof}
	Without loss of generality. we can assume that $a<b<c$. Let $y=\max\{\frac{c}{b},\frac{b}{a}\}$. Take $x$ from  $(y,\frac{c}{a})$ arbitrarily. And we take $l$ from $(xy,x\cdot\frac{c}{a})$ arbitrarily. Then $1<y<x<l$.
	
	Now, we define a $2$-coloring of $(0,\infty)$ by the following:
	
	For any $z\in (0,\infty)$, if there exists $m\in \N$ such that $z\in [l^{m},x\cdot l^{m})$, then we assign color $1$ to $z$; if If there exists $m\in \N$ such that $z\in [x\cdot l^{m},l^{m+1})$, then we assign color $2$ to $z$; otherwise, we assign color $1$ to $z$.
	
	It is easy to see that $\N$ can inherite a $2$-coloring, denoted by $\phi$, from the above coloring of $(0,\infty)$.
	
		Take $n$ from $\N$ arbitrarily and fix it. Assume that there exists a strictly increasing sequence $\{k_i\}_{i\ge 1}\subset \N$ such that $\phi$ takes constant value on the set $$\{n+ak_1,n+ak_2,\ldots\}\cup \{n+bk_1,n+bk_2,\ldots\}\cup \{n+ck_1,n+ck_2,\ldots\}.$$
		
		Take sufficiently large $k$ from $\{k_i:i\ge 1\}$ such that $k>l$ and for any natural numbebr $m$ with $l^{m+1}>k$, we have that $(\frac{c}{a}-x)l^m>(\frac{c}{a}-1)n$ and  $(\frac{c}{a}x-l)l^m>(\frac{c}{a}-1)n$. Then:
		\begin{itemize}
			\item If $\phi(n+ak)=\phi(n+bk)=1$, then there exists $m\in \N$ such that $$n+ak,n+bk\in [l^{m},x\cdot l^{m}).$$
		   	Then we have that $$\frac{c}{a   }(l^{m}-n)+n\le n+ck<\frac{c}{b}(x\cdot l^{m}-n)+n.$$ Note that $$\frac{c}{b}(x\cdot l^{m}-n)+n<\frac{c}{b}x\cdot l^{m}<l^{m+1}\ \text{and}$$ $$\frac{c}{a   }(l^{m}-n)+n=xl^{m}+(\frac{c}{a}-x)l^{m}+n-\frac{cn}{a}>xl^{m}.$$ Then $\phi(n+ck)=2$. It is a contradiction.
			\item If $\phi(n+ak)=\phi(n+bk)=2$, then there exists $m\in \N$ such that $$n+ak,n+bk\in [x\cdot l^{m},l^{m+1}).$$ A simple calculation gives $$\frac{c}{a}(x\cdot l^{m}-n)+n\le n+ck<\frac{c}{b}(l^{m+1}-n)+n.$$
			Note that  $$\frac{c}{b}(l^{m+1}-n)+n<\frac{c}{b}l^{m+1}<x\cdot l^{m+1}\ \text{and}$$ $$\frac{c}{a}(x\cdot l^{m}-n)+n=l^{m+1}+(\frac{c}{a}x-l)l^{m}+n-\frac{cn}{a}>l^{m+1}.$$ Then $\phi(n+ck)=1$. It is a contradiction.
		\end{itemize}
		To sum up, for any $i\in \{1,2\}$ and any $n\in \N$, the set $$\{m\in \N:\phi(n+am)=\phi(n+bm)=\phi(n+cm)=i\}$$ is empty or finite. This finishes the proof.
\end{proof}

\bibliographystyle{plain}
\normalem
\bibliography{ref}

\end{document}